\documentclass[11pt]{amsart}
\usepackage{amsmath}
\usepackage{amssymb}
\usepackage{amsfonts}
\usepackage{geometry}
\usepackage{leqno}
\usepackage{mathtools}
\usepackage{float}
\usepackage{tabu}
\usepackage{booktabs}
\usepackage{array, multicol, multirow, makecell}
\usepackage{soul}
\usepackage{amsthm}
\usepackage[shortlabels]{enumitem}
\usepackage{xcolor}
\usepackage[all]{xy}
\usepackage{nicefrac}
\definecolor{cobalt}{RGB}{61,89,171}
\definecolor{DarkViolet}{RGB}{148,0,211}
\usepackage[colorlinks,citecolor=DarkViolet,linkcolor=DarkViolet,urlcolor=DarkViolet,pdfpagemode=UseNone,backref = page]{hyperref}

\newcommand\frg{{\mathfrak g}}
\newcommand\frh{{\mathfrak h}}
\newcommand\frn{{\mathfrak n}}
\newcommand\frX{{\mathfrak X}}

\newcommand{\R}{\mathbb{R}}
\newcommand{\Z}{\mathbb{Z}}

\newcommand{\G}{\mathrm{G}}
\renewcommand{\d}{\mathrm{d}}
\newcommand{\tr}[1]{\mathrm{#1}}
\newcommand{\calM}{\mathcal{M}}
\newcommand{\calO}{\mathcal{O}}
\newcommand{\calT}{\mathcal{T}}
\DeclareMathOperator{\Aut}{Aut}
\DeclareMathOperator{\Der}{Der}
\DeclareMathOperator{\Diff}{Diff}

\newcommand{\GL}{\operatorname{GL}}
\newcommand{\SO}{\operatorname{SO}}

\newtheorem{theorem}{Theorem}[section]
\newtheorem*{theorem*}{Theorem}

\newtheorem{corollary}[theorem]{Corollary}
\newtheorem{proposition}[theorem]{Proposition}

\newtheorem{example}[theorem]{Example}

\newtheorem{definition}[theorem]{Definition}

\theoremstyle{remark}
\newtheorem{remark}[theorem]{Remark}

\usepackage{listings}

\definecolor{mygreen}{rgb}{0,0.6,0}
\definecolor{mygray}{rgb}{0.5,0.5,0.5}
\definecolor{mymauve}{rgb}{0.58,0,0.82}

\title{Moduli spaces of (co)closed $\G_2$-structures on nilmanifolds}

\author[G. Bazzoni]{Giovanni Bazzoni}
\address{Dipartimento di Scienza ed Alta Tecnologia, Universit\`a degli Studi dell'Insubria, Via Valleggio 11, 22100, Como, Italy}
\email{giovanni.bazzoni@uninsubria.it}

\author[A. Gil-García]{Alejandro Gil-García}
\address{Fachbereich Mathematik, Universit\"at Hamburg, Bundesstra\ss e 55, 20146 Hamburg, Germany}
\email{alejandro.gil.garcia@uni-hamburg.de}

\begin{document}


\maketitle

\begin{abstract}
We compute the dimensions of some moduli spaces of left-invariant closed and coclosed $\G_2$-structures on 7-dimensional nilmanifolds, showing that they are not related to the third Betti number. We also prove that, in contrast to the case of closed $\G_2$-structures, the group of automorphisms of a coclosed $\G_2$-structure is not necessarily abelian.
\end{abstract}


\section{Introduction and Preliminaries}\label{sec:Preliminaries}

A $\G_2$-structure on a 7-dimensional manifold is a reduction of the structure group of the frame bundle from $\GL(7,\R)$ to $\G_2\subset \SO(7)$; hence, such a manifold is oriented and Riemannian. The existence of a $\G_2$-structure is topological, and equivalent to the vanishing of the first and the 
second Stiefel-Whitney class of the manifold (thus to orientability and spinnability); this is due to Gray \cite{Gray}. Since $\G_2$ is the stabilizer of a (positive) 3-form in $\R^7$, a $\G_2$-structure on a 7-manifold $M$ can equivalently be defined as a positive 3-form $\varphi\in\Omega^3_+(M)$, one for which it is possible to find, at each point $p\in M$, local coordinates $\{x^1,\ldots,x^7\}$ such that
\[
\varphi_p=\d x^{127}+\d x^{347}+\d x^{567}+\d x^{135}-\d x^{236}-\d x^{146}-\d x^{245}\,,
\]
where $\d x^{ijk}=\d x^i\wedge \d x^j\wedge \d x^k$. When one thinks of a $\G_2$-structure in this way, the metric and the orientation are usually denoted $g_\varphi$ and $\tr{vol}_\varphi$, since they are determined by the positive 3-form $\varphi$. In fact, $g_\varphi$ and $\tr{vol}_\varphi$ can be recovered from $\varphi$ via the formula
\[
6g_\varphi(X,Y)\tr{vol}_\varphi=\iota_X\varphi\wedge\iota_Y\varphi\wedge\varphi\,, \quad X,Y\in\frX(M)\,.
\]

\medskip

$\G_2$-structures are related to manifolds with Riemannian holonomy contained in $\G_2$: by a classical result of Fernández and Gray \cite{FG}, the metric $g_\varphi$ has holonomy contained in $\G_2$ if and only if $\d\varphi=0=\d*_\varphi\varphi$, namely, if $\varphi$ is both {\em closed} and {\em coclosed}. Notice that $\G_2$ is one of the exceptional holonomy groups appearing in the celebrated Berger list \cite{Berger}.

\medskip

Furthermore, the differential and the codifferential of a $\G_2$-form $\varphi$ on a 7-manifold $M$ are governed by the so-called {\em torsion forms} $\tau_i\in\Omega^i(M)$, $i=0,1,2,3$, according to the following equations (see \cite[Proposition~1]{Bryant}):
\begin{equation}\label{eq:pure_classes}
\left\{\begin{array}{ccl}
    \d\varphi & = & \tau_0\psi+3\tau_1\wedge\varphi+*_\varphi\tau_3  \\
    \d\psi & = & 4\tau_1\wedge \psi+\tau_2\wedge\varphi
\end{array}\right.    
\end{equation}
where $\psi=*_\varphi\varphi$. Thus, there are 16 classes of $\G_2$-structures, and the holonomy reduction is equivalent to the vanishing of all torsion forms, so that Riemannian manifolds with holonomy contained in $\G_2$ are also known as torsion-free $\G_2$-manifolds. The {\em pure} classes are those for which all torsion forms vanish, but one. Thus, looking at \eqref{eq:pure_classes}, we get four pure classes:
\begin{table}[h!]
\begin{center}
{\tabulinesep=1.2mm
\begin{tabu}{lllccccc}
\toprule[1.5pt]
$\tau_i=0, i=1,2,3$, $\tau_0\neq 0$ & nearly parallel & $\d\varphi=\tau_0*_\varphi\varphi$\\
\specialrule{1pt}{0pt}{0pt}
$\tau_i=0, i=0,2,3$, $\tau_1\neq 0$ & 
locally conformally parallel & $\d\varphi=3\tau_1\wedge\varphi$, $\d\psi=4\tau_1\wedge\psi$\\
\specialrule{1pt}{0pt}{0pt}
$\tau_i=0, i=0,1,3$, $\tau_2\neq 0$ & closed & $\d\psi=\tau_2\wedge\varphi$\\
\specialrule{1pt}{0pt}{0pt}
$\tau_i=0, i=0,1,2$, $\tau_3\neq 0$ & purely coclosed & $\d\varphi=*_\varphi\tau_3\Leftrightarrow \varphi\wedge \d\varphi=0$\\
\bottomrule[1pt]
\end{tabu}}
\end{center}
\end{table}

Notice that coclosed $\G_2$-structures, those for which $\d\psi=0$, form a class which contains both nearly parallel and purely coclosed $\G_2$-structures. Since one has $\tau_0=\frac{1}{7}*_\varphi(\varphi\wedge\d \varphi)$, purely coclosed $\G_2$-structures are precisely the coclosed $\G_2$-structures for which $\tau_0=0$.

\medskip

In this paper we are mainly interested in left-invariant $\G_2$-structures on nilmanifolds. A {\em nilmanifold} is the quotient of a connected, simply connected, nilpotent Lie group $G$ by a lattice $\Gamma$, i.e.\ a discrete, cocompact subgroup. Recall that a Lie group $G$ is nilpotent if and only if its Lie algebra $\frg$ is. Nilpotency of a Lie algebra $\frg$ is defined as follows: the {\em lower central series} of a Lie algebra $\frg$ is defined by $\frg_0=\frg$ and $\frg_i=[\frg_{i-1},\frg]$, for $i\geq 1$. A Lie algebra $\frg$ is {\em nilpotent} if there exists $m$ such that $\frg_m=0$. The smallest such $m$ is the {\em nilpotency step} of a nilpotent Lie algebra.

\medskip

Let $G$ be a connected, simply connected, nilpotent Lie group. Tensors on $\frg$ correspond to left-invariant tensors on $G$, which descend to tensors on the quotient nilmanifold $\Gamma\backslash G$; with a slight abuse of notation, we call such tensors on $\Gamma\backslash G$ left-invariant. For instance, left-invariant forms know a great deal about the topology of a nilmanifold: Nomizu \cite{Nomizu} proved that $H^*_{dR}(\Gamma\backslash G)$ is isomorphic to the Lie algebra cohomology of $\frg$. A $\G_2$-structure on a nilmanifold $\Gamma\backslash G$ is called {\em left-invariant} if it is induced by a left-invariant 3-form $\varphi\in\Omega^3(\Gamma\backslash G)$; hence, such a $\G_2$-structure can be described by an element $\varphi\in\Lambda^3\frg^*$. 

\medskip

Nilmanifolds are parallelizable, so that they admit $\G_2$-structures, and also coclosed ones, by a result of Crowley and Nordström \cite{CrNs}; not all of them are left-invariant. It is known that nilmanifolds (not tori) can not admit torsion-free $\G_2$-structures. As for pure, left-invariant $\G_2$-structures on nilmanifolds, we have the following facts:
\begin{itemize}
    \item they can not be nearly parallel; this is because the metric of a nearly parallel $\G_2$-structure is Einstein \cite[Proposition~3.10]{FKMS}. By a result of Milnor, nilmanifolds (not tori) do not admit left-invariant Einstein metrics \cite[Theorem~2.4]{Milnor}.
    \item they can not be locally conformally parallel; by \cite[Theorem~A]{IPP}, compact manifolds with locally conformally parallel $\G_2$-structures are mapping tori of simply connected nearly Kähler 6-manifolds, hence they have $b_1=1$. However, a nilmanifold has $b_1\geq 2$.
\end{itemize}

We are therefore naturally led to consider left-invariant closed and (purely) coclosed $\G_2$-structures on nilmanifolds. The classification of 7-dimensional nilpotent Lie algebras which admit closed $\G_2$-structures was obtained by Conti and Fernández \cite{CF11}. In \cite{BFF18} Bagaglini, Fino and Fernández considered coclosed $\G_2$-structures on decomposable and 2-step nilpotent 7-dimensional Lie algebras (a Lie algebra is {\em decomposable} if it is a direct sum of two non-trivial Lie algebras). Purely coclosed $\G_2$-structures on 2-step nilpotent Lie algebras have been studied in \cite{dBMR22}. Purely coclosed $\G_2$-structures on nilpotent Lie algebras which are either decomposable or have nilpotency step $\leq 4$ have been tackled in \cite{BGM21}.

\medskip

The main focus of this paper is on moduli spaces of closed and coclosed $\G_2$-structures on nilpotent Lie algebras, and of (left-invariant) closed and coclosed $\G_2$-structures on nilmanifolds (we refer to Sections \ref{Sec:3} and \ref{Sec:4} for all the relevant definitions). The starting points of our analysis are the following:
\begin{itemize}
    \item by a result of Joyce \cite{Joyce}, the moduli space of torsion-free $\G_2$-structures on a compact manifold $M$ is a smooth manifold of dimension $b_3(M)$;
    \item it is known that, on a compact manifold, both the moduli space of closed $\G_2$-structures and that of coclosed $\G_2$-structures are infinite-dimensional; we include a proof of these facts in Sections \ref{Sec:3} and \ref{Sec:4}.
\end{itemize}

In Section~\ref{Sec:3}, we compute the dimension of the moduli space $\calM^{\tr{c}}(\frg)$ of closed $\G_2$-structures on a 7-dimensional nilpotent Lie algebra $\frg$ which admits such a structure, and of the moduli space of left-invariant closed $\G_2$-structures on any nilmanifold $\Gamma\backslash G$, $\calM^{\tr{c}}_{\tr{inv}}(\Gamma\backslash G)$. Furthermore, we prove that $\calM^{\tr{c}}(\frg)$ embeds into $\calM^{\tr{c}}(\Gamma\backslash G)$, the moduli space of closed $\G_2$-structures on $\Gamma\backslash G$, thus the former can be considered as a finite-dimensional approximation of the latter. Finally, we show that the dimension of $\calM^{\tr{c}}(\frg)$ is not related to $b_3(\frg)=b_3(\Gamma\backslash G)$.

In Section~\ref{Sec:4} we follow the same train of thoughts in the coclosed setting; the results we get are tantamount to those obtained in Section~\ref{Sec:3}. As an application of our study, we show that the automorphism group of a coclosed $\G_2$-structure on a compact manifold need not be abelian. This differs from the closed situation; indeed, as proved by Podestà and Raffero in \cite{PR19}, the automorphism group of a closed $\G_2$-structure on a compact manifold is abelian.

The appendices contain an overview of the computations, which were executed with \texttt{SageMath} \cite{sagemath}.

\medskip

\noindent {\bf Acknowledgements.} We are both very thankful to Vicente Cortés for suggesting this problem and for useful conversations. We are indebted to Anna Fino and Vicente Muñoz for valuable comments on a preliminary version of this paper. Finally, we would like to thank both anonymous referees for their comments, from which the paper benefitted a great deal. G.~B.~is partially supported by the PRIN2022 project ``Interactions between Geometric Structures and Function Theories'', by GNSAGA of INdAM and by grants PID2020-118452GB-I00 and  PID2021-126124NB-I00 (MCIN/AEI/10.13039/501100011033). A.~G.-G.~is supported by the German Science Foundation (DFG) under Germany's Excellence Strategy  --  EXC 2121 ``Quantum Universe'' -- 390833306.

\section{Moduli spaces of closed $\G_2$-structures on nilmanifolds}\label{Sec:3}

In this section we consider three moduli spaces of closed $\G_2$-structures:
\begin{itemize}
    \item closed $\G_2$-structures on (nilpotent) Lie algebras,
    \item left-invariant closed $\G_2$-structures on nilmanifolds,
    \item closed $\G_2$-structures on 7-manifolds,
\end{itemize}
and prove some relations among them. Given a 7-dimensional Lie algebra $\frg$ with a closed $\G_2$-structure, we set
\[
V^{\tr{c}}(\frg)\coloneqq \left\{\varphi\in\Lambda^3\frg^*\mid \text{$\varphi$ is a $\G_2$-structure and } \d\varphi=0\right\}\,.
\]
Thus $\varphi\in\Lambda^3\frg^*$ is a closed 3-form and the formula
\[
6g_\varphi(X,Y)\operatorname{vol}_\varphi\coloneqq \iota_X\varphi\wedge\iota_Y\varphi\wedge\varphi
\]
defines a scalar product $g_\varphi$ and a volume $\operatorname{vol}_\varphi$ on $\frg$. Since being positive-definite is an open condition, $V^{\tr{c}}(\frg)$ is an open subset of $Z^3(\frg^*)$, and its dimension coincides with that of $Z^3(\frg^*)$. $\Aut(\frg)$, the group of Lie algebra automorphisms of $\frg$, is a closed subgroup of $\GL(\frg)$, and acts naturally on $V^{\tr{c}}(\frg)$; we denote by $\calO(\varphi)\cong\Aut(\frg)/\Aut(\frg)_\varphi $ the orbit of the $\Aut(\frg)$-action passing through $\varphi$; here 
\[
\Aut(\frg)_\varphi\coloneqq \Aut(\frg)\cap\{A\in\GL(\frg) \mid A^*\varphi=\varphi\}\,,
\] 
which is a compact subgroup.

\begin{definition}
The {\em moduli space of closed $\G_2$-structures on $\frg$} is
\[
 \calM^{\tr{c}}(\frg)\coloneqq V^{\tr{c}}(\frg)/\Aut(\frg)=\{\calO(\varphi)\mid\varphi\in V^c(\frg)\}\,.
\]
\end{definition}

$\calM^{\tr{c}}(\frg)$ is not smooth in general, rather a stratified space, since the orbits of different elements of $V^{\tr{c}}(\frg)$ under the action of $\Aut(\frg)$ may have different dimensions. 

\begin{definition}
The {\em dimension} of $\calM^{\tr{c}}(\frg)$ is 
\begin{equation}\label{eq:dimension}
\dim\calM^{\tr{c}}(\frg)\coloneqq\dim V^{\tr{c}}(\frg)-\dim\calO(\varphi)=\dim V^{\tr{c}}(\frg)-(\dim\Aut(\frg)-\dim\Aut(\frg)_\varphi)\,,    
\end{equation}
where $\calO(\varphi)$ is a principal orbit, i.e.\ one of maximal dimension; equivalently, the stabilizer of $\varphi$ has minimal dimension.    
\end{definition}

On a nilmanifold $\Gamma\backslash G$ one has special kinds of diffeomorphisms, which are called {\em $G$-induced} in the literature: these are precisely automorphisms of $G$ which preserve the lattice $\Gamma$; we denote them by $\Aut(G,\Gamma)$. Thus, a meaningful definition of {\em moduli space of left-invariant closed $\G_2$-structures} on $\Gamma\backslash G$ is
\[
\calM^{\tr{c}}_{\mathrm{inv}}(\Gamma\backslash G)\coloneqq V^{\tr{c}}_{\mathrm{inv}}(\Gamma\backslash G)/\Aut(G,\Gamma)=\{\calO(\varphi_{\mathrm{inv}})\mid\varphi_{\mathrm{inv}}\in V^{\tr{c}}_{\mathrm{inv}}(\Gamma\backslash G)\}\,,
\]
where $V^{\tr{c}}_{\mathrm{inv}}(\Gamma\backslash G)$ is the space of left-invariant closed $\G_2$-structures on $\Gamma\backslash G$, so that $V^{\tr{c}}_{\mathrm{inv}}(\Gamma\backslash G)\cong V^{\tr{c}}(\frg)$. According to \cite[Proposition 1.1]{MoMo}, $\Aut(G,\Gamma)$ is a discrete subgroup of $\Aut(G)$, hence 
\begin{equation}\label{eq:dimension_inv}
\dim\calM^{\tr{c}}_{\mathrm{inv}}(\Gamma\backslash G)=\dim V^{\tr{c}}(\frg)=\dim\left(\ker\left(\d\colon\Lambda^3\frg^*\to\Lambda^4\frg^*\right)\right)\geq b_3(\frg)=b_3(\Gamma\backslash G)\,.
\end{equation}



A similar dimension count appears, in the context of symplectic structures on nilpotent Lie algebras and nilmanifolds, in \cite[Equation~31]{Sal01}. We can compare $\calM^{\tr{c}}_{\mathrm{inv}}(\Gamma\backslash G)$ with the space of {\em all} closed $\G_2$-structures on $\Gamma\backslash G$. More generally, we can give the following definition.

\begin{definition}
Let $M$ be a manifold endowed with closed $\G_2$-structures. The {\em moduli space of closed $\G_2$-structures on $M$} is
\[
\calM^{\tr{c}}(M)\coloneqq\{\varphi\in\Omega^3_+(M), \d\varphi=0\}/\Diff_0(M)=\{\calO(\varphi) \mid \varphi\in\Omega^3_+(M), \d\varphi=0\}\,,
\]
where $\calO(\varphi)$ is the orbit of $\varphi$ under the action of $\Diff_0(M)$, so that $\calO(\varphi)\cong\Diff_0(M)/\Diff_0(M)_{\varphi}$.
\end{definition}

Here $\Diff_0(M)_{\varphi}$ is the stabilizer of $\varphi$ in $\Diff_0(M)$. The following result is well-known, but we include a proof for completeness. We are grateful to Lothar Schiemanowski for pointing this argument out to us.

\begin{proposition}
Let $M$ be a compact manifold endowed with closed $\G_2$-structures. Then $\calM^{\tr{c}}(M)$ is infinite-dimensional.
\end{proposition}

\begin{proof}
 $\{\varphi\in\Omega^3_+(M), \d\varphi=0\}$ is an open set in the infinite-dimensional vector space $Z^3(M)$ of closed 3-forms on $M$, acted on by $\Diff_0(M)$. Pick $[\varphi]\in\calM^{\tr{c}}(M)$. Then
 \[
 T_{[\varphi]}\calM^{\tr{c}}(M)=Z^3(M)/T_\varphi(\Diff_0(M)\cdot\varphi)\,,
 \]
 where $T_\varphi(\Diff_0(M)\cdot\varphi)=\{\mathcal L_X\varphi \mid X\in\frX(M)\}$. Since $\d\varphi=0$,
 \[
 T_\varphi(\Diff_0(M)\cdot\varphi)=\{\d(\iota_X\varphi) \mid X\in\frX(M)\}=\d\left(\Omega^2_7(M)\right)\,,
 \]
 where $\Omega^2(M)=\Omega^2_7(M)\oplus\Omega^2_{14}(M)$ under the natural action of $\G_2$. By harmonic theory, $Z^3(M)=\mathcal{H}^3(M)\oplus \d\left(\Omega^2(M)\right)$, hence $T_{[\varphi]}\calM^{\tr{c}}(M)\cong \mathcal{H}^3(M)\oplus \d\left(\Omega^2_{14}(M)\right)$, which is infinite-dimensional.
\end{proof}

\begin{remark}
The group $\Diff_0(M)_\varphi$ is non-trivial, in general. As proved in \cite[Theorem~2.1]{PR19}, it is an abelian Lie group with $\dim\Diff_0(M)_\varphi\leq\min\{b_2(M),6\}$.
\end{remark}

Closed $\G_2$-structures are sometimes considered the symplectic analogue of $\G_2$-structures. Indeed, they are both defined by closed differential forms satisfying the same non-degeneracy condition, namely that at each point of the manifold their orbit under the action of the general linear group in the corresponding tangent space is open; in other words, they are {\em stable forms} in the sense of Hitchin \cite{Hitchin}. If $M$ is a compact symplectic manifold, let us denote by $\mathcal{S}(M)$ the space of symplectic structures on $M$; it is an open subset of $Z^2(M)$, acted on by the diffeomorphism group $\Diff(M)$ of $M$ and by its identity component $\Diff_0(M)$. Setting
\[
\calM^{\mathrm{s}}(M)\coloneqq \mathcal{S}(M)/\Diff_0(M)\,,
\]
one has that $\calM^{\mathrm{s}}(M)$ is a smooth manifold of dimension $b_2(M)$ (see for instance \cite[Proposition~3.1]{FH02}). With respect to this, closed $\G_2$-structures and symplectic structures behave quite differently.


\medskip

If $M=\Gamma\backslash G$ is a nilmanifold admitting closed $\G_2$-structures, we have a natural inclusion $\imath\colon V^{\tr{c}}(\frg)\hookrightarrow\{\varphi\in\Omega^3_+(M) \mid \d\varphi=0\}$.

\begin{theorem}\label{thm:closed}
    $\imath$ induces an inclusion $\bar{\imath}\colon \calM^{\tr{c}}(\frg)\hookrightarrow\calM^{\tr{c}}(M)$.    
\end{theorem}

\begin{proof}
    Consider $\varphi_1,\varphi_2\in V^{\tr{c}}(\frg)$, and view them as elements of $\{\varphi\in\Omega^3_+(M) \mid \d\varphi=0\}$ via $\imath$. Suppose that they define the same class in $\calM^{\tr{c}}(M)$. We must show that $[\varphi_1]=[\varphi_2]\in\calM^{\tr{c}}(\frg)$. By hypothesis, there exists $f\in\Diff_0(M)$ such that $f^*\varphi_2=\varphi_1$. Since the projection $\pi\colon G\to M$ is the universal cover, we can lift $f$ to a diffeomorphism $\tilde{f}\colon G\to G$ such that $\tilde{f}^*\tilde{\varphi}_2=\tilde{\varphi}_1$; here $\tilde{\varphi}_j$, $j=1,2$, denotes the left-invariant form on $G$ which projects to $\varphi_j$. Notice that $\tilde{\varphi}_j$ induces a left-invariant Riemannian metric $\tilde{g}_j$ on $G$, $j=1,2$, and that $\tilde{f}\colon (G,\tilde{g}_1)\to(G,\tilde{g}_2)$ is an isometry. Composing with a left-translation if necessary, we can suppose that $\tilde{f}$ fixes the identity of $G$. By a theorem of Wilson \cite{Wilson} (see also \cite[Proposition~1.3]{Lauret}), it follows that $\tilde{f}\in\tr{Aut}(G)\cong\Aut(\frg)$. But this implies that $[\varphi_1]=[\varphi_2]$ in $\calM^{\tr{c}}(\frg)$.
\end{proof}

\begin{remark}
    Wilson's theorem \cite[Theorem~3]{Wilson} was generalized to the case of completely solvable and unimodular Lie algebras \cite[Theorem~4.3]{GW88} (see also \cite[Theorem~4.2]{KL21}). Hence Theorem~\ref{thm:closed} (and Theorem~\ref{thm:coclosed} below) holds also in this broader class of Lie algebras.
\end{remark}

Let $M$ be a $\G_2$-manifold, that is, a 7-manifold with torsion-free $\G_2$-structures, and consider the {\em moduli space of torsion-free $\G_2$-structures} on $M$,
\[
\calM(M)\coloneqq\calT(M)/\Diff_0(M)\,,
\]
where 
\[
\calT(M)=\{\varphi\in\Omega^3_+(M) \mid \d\varphi=0=\d*_\varphi\varphi\}\,.
\]
The following theorem, due to Joyce, summarizes the properties of $\calM(M)$ (see for instance \cite[Theorem~7.5]{Karigiannis}):

\begin{theorem}\label{thm:Joyce_theorem}
Let $M$ be a compact $\G_2$-manifold. Then $\calM(M)$ is a smooth manifold of dimension $b_3(M)$. In fact, the period map $\calM(M)\to H^3(M;\R)$, $[\varphi]\mapsto [\varphi]_{dR}$ is a local diffeomorphism.
\end{theorem}

\begin{remark}
    According to Theorem~\ref{thm:Joyce_theorem} the moduli space of torsion-free $\G_2$-structures on the 7-dimensional torus $\mathbb{T}^7$ has dimension $b_3(\mathbb{T}^7)=35$. On the other hand $\dim\calM^{\tr{c}}_{\mathrm{inv}}(\mathbb{T}^7)=\dim V^{\tr{c}}(\R^7)=35=b_3(\mathbb{T}^7)$. Note that in this case a left-invariant closed $\G_2$-structure on the torus is also coclosed (and vice versa), hence torsion-free by \cite{FG}.
\end{remark}

If $M=\Gamma\backslash G$ is a nilmanifold, Theorem~\ref{thm:closed} shows that $\calM^{\tr{c}}(M)$ contains a full copy of $\calM^{\tr{c}}(\frg)$. We can ask whether this ``finite-dimensional approximation'' of $\calM^{\tr{c}}(M)$ has some relationship with $b_3(M)=b_3(\frg)$. We address this question next.


\medskip

According to \eqref{eq:dimension}, to compute $\dim \calM^{\tr{c}}(\frg)$ we need to compute $\dim V^{\tr{c}}(\frg)$, $\dim\Aut(\frg)$ and $\dim\Aut(\frg)_\varphi$.

\begin{itemize}
    \item As we remarked above, $\dim V^{\tr{c}}(\frg)=\dim Z^3(\frg^*)$;
    \item to compute $\dim\Aut(\frg)$, we simply compute the dimension of its Lie algebra $\Der(\frg)=\{D\in\mathfrak{gl}(\frg) \mid D[X,Y]=[DX,Y]+[X,DY] \ \forall~X,Y\in\frg\}$;
    \item to compute $\dim\Aut(\frg)_\varphi$ it is enough to calculate the dimension of its Lie algebra $\Der(\frg)_\varphi=\{D\in\Der(\frg) \mid D\cdot\varphi=0\}$, where, for $X,Y,Z\in\frg$,
\[
(D\cdot \varphi)(X,Y,Z)=-\varphi(DX,Y,Z)-\varphi(X,DY,Z)-\varphi(X,Y,DZ)\,.
\]
\end{itemize}

To compute the dimension of $\Aut(\frg)_\varphi$ we need to start from a background closed $\G_2$-structure on $\frg$. In \cite{CF11}, Conti and Fernández classified the 7-dimensional NLAs admitting a closed $\G_2$-structure: they showed that only 12 of the 7-dimensional NLAs admit such a structure; 3 of them are decomposable and 9 are indecomposable. For each NLA $\frg$ with a closed $\G_2$-structure $\varphi$, the authors exhibit an explicit coframe $\{\eta^1,\ldots,\eta^7\}$ of $\frg^*$ on which $\varphi$ takes the standard form,
\[
\varphi=\eta^{127}+\eta^{347}+\eta^{567}+\eta^{135}-\eta^{236}-\eta^{146}-\eta^{245}\,.
\]

With this background closed $\G_2$-structure $\varphi$ we can compute the dimension of the stabilizer $\Aut(\frg)_\varphi$ for each of the pairs $(\frg,\varphi)$; as we explained before, to do this, we compute the dimension of its Lie algebra, $\Der(\frg)_\varphi$. The results are contained in the following table (see Appendix~\ref{appendix:1} for the explicit closed $\G_2$-form we use and for an overview of the computations).


	
	

\begin{table}[H]
\centering
\caption{Nilpotent Lie algebras with closed $\G_2$-structures}\label{Table:1}
\vspace{0.25 cm}
{\tabulinesep=1.2mm
\begin{tabu}{c|c|c|c|c|}
\toprule[1.5pt]
\multirow{2}{*}{$\frg$} &  & $\dim $ & $\dim$ & $\dim$\\
 &  & $V^{\tr{c}}(\frg)$ & $\Aut(\frg)$ & $\Aut(\frg)_\varphi$\\
\specialrule{1pt}{0pt}{0pt}
$\frn_1$ & $(0,0,0,0,0,0,0)$ & 35 & 49 & 14\\
\specialrule{1pt}{0pt}{0pt}
$\frn_4$ & $(0,0,0,0,12,13,0)$ & 27 & 27 & 1\\
\specialrule{1pt}{0pt}{0pt}
$\frn_{10}$ & $(0,0,0,12,13,23,0)$ & 27 & 25 & 1\\
\specialrule{1pt}{0pt}{0pt}
$147E1(2)$ & $\begin{array}{c}
            \left(0,0,0,12,23,-13,\right.\\
            \left.2\cdot(-16+25+26-34)\right)\end{array}$ & 20 & 15 & 1\\
\specialrule{1pt}{0pt}{0pt}
$247A$ & $(0,0,0,12,13,14,15)$ & 23 & 19 & 0\\
\specialrule{1pt}{0pt}{0pt}
$247L$ & $(0,0,0,12,13,14+23,15)$ & 23 & 17 & 0\\
\specialrule{1pt}{0pt}{0pt}
$257A$ & $(0,0,12,0,0,13+24,15)$ & 23 & 19 & 0\\
\specialrule{1pt}{0pt}{0pt}
$257B$ & $(0,0,12,0,0,13,14+25)$ & 23 & 18 & 0\\
\specialrule{1.25pt}{0pt}{0pt}
$1357N(1)$ & $(0,0,12,0,13+24,14,15+23+34+46)$ & 20 & 13 & 0\\
\specialrule{1pt}{0pt}{0pt}
$1357S(-3)$ & $\begin{array}{c}
            \left(0,0,12,0,13,23+24,\right.\\
            \left.15+16+25-3\cdot 26+34\right)\end{array}$ & 19 & 12 & 0\\
\specialrule{1pt}{0pt}{0pt}
$12457H$ & $(0,0,12,13,23,15+24,16+34)$ & 19 & 12 & 0\\
\specialrule{1pt}{0pt}{0pt}
$12457I$ & $(0,0,12,13,23,15+24,16+25+34)$ & 19 & 11 & 0\\
\bottomrule[1pt]
\end{tabu}}
\end{table}


\begin{remark}
Throughout the paper we use the so-called Salamon's notation for Lie algebras. For instance, $(13,-23,0)$ denotes the 3-dimensional Lie algebra whose dual space has a basis $\{e^1,e^2,e^3\}$ with $\d e^1=e^{13}$, $\d e^2=-e^{23}$ and $\d e^3=0$. Here $\d$ denotes the Chevalley-Eilenberg differential, which is the dual map to the Lie bracket.
\end{remark}

\begin{remark}
If $\frg=\frn_1$, the abelian Lie algebra, the automorphism group acts transitively on the space of (closed) 3-forms; the stabilizer of a point is isomorphic to $\G_2$, and the moduli space is just one point.    
\end{remark}

\begin{remark}
Notice that, according to \eqref{eq:dimension_inv}, the dimension of $V^{\textrm{c}}(\frg)$ in Table \ref{Table:1} coincides with the dimension of $\mathcal{M}^{\textrm{c}}_{\textrm{inv}}(M)$, where $M$ is any nilmanifold quotient of $G$. 
\end{remark}


As we see from Table~\ref{Table:1}, for most of the pairs $(\frg,\varphi)$ one has $\dim\Aut(\frg)_\varphi=0$, the stabilizer is minimal, and the orbit $\calO(\varphi)$ has maximal dimension. Hence we can immediately compute the dimension of the corresponding moduli space. This is not the case, however, for $\frg\in\{\frn_4,\frn_{10},147E1(2)\}$. To determine whether the dimension of the given orbit is maximal in each case, the strategy is to slightly perturb the given $\varphi$ to a closed $\G_2$ form $\varphi_\varepsilon$, for $\varepsilon$ sufficiently small, whose stabilizer has dimension 0.
	
\begin{itemize}
	\item In the Lie algebra $\frn_4$, the orbit of the given $\varphi$ seems to be minimal.
	\item In the Lie algebra $\frn_{10}$ the orbit of the given $\varphi$ has 1-dimensional stabilizer. Considering the closed $\G_2$-structure 
    $\varphi_\varepsilon=\varphi+\varepsilon(e^{257}+e^{347})$, one can check that $\dim\Aut(\frn_{10})_{\varphi_\varepsilon}=0$.
	\item In the Lie algebra $\frg=147E1(2)$ the orbit of the given $\varphi$ has 1-dimensional stabilizer. Considering the closed $\G_2$-structure $\varphi_\varepsilon=\varphi+\varepsilon(e^{156}-e^{237}+e^{346})$, one sees that $\dim\Aut(\frg)_{\varphi_\varepsilon}=0$.
\end{itemize}

Summing up, we can compute the dimension of $\calM^{\mathrm{c}}(\frg)$ for most of the 7-dimensional nilpotent Lie algebras with a closed $\G_2$ structure. We also write down the dimension of $\calM^{\tr{c}}_{\mathrm{inv}}(M)$, where $M$ is any nilmanifold quotient of $G$, the only connected, simply connected Lie group with Lie algebra $\frg$; this is simply $\dim V^{\tr{c}}(\frg)$. We collect the results in the following table, together with the third Betti number of each Lie algebra:

\begin{table}[H]
\centering
\caption{}\label{Table:2}
\vspace{0.25 cm}
{\tabulinesep=1.2mm
\begin{tabu}{c|c|c||c|c|c|}
\toprule[1.5pt]
$\frg$ & $\dim\calM^{\mathrm{c}}(\frg)$ & $b_3(\frg)$ & $\frg$ & $\dim\calM^{\mathrm{c}}(\frg)$ & $b_3(\frg)$\\
\specialrule{1pt}{0pt}{0pt}
$\frn_1$ & 0 & 35 & $257A$ & 4 & 14 \\
\specialrule{1pt}{0pt}{0pt}
$\frn_4$ & $\leq 1$ & 21 & $257B$ & 5 & 13  \\
\specialrule{1pt}{0pt}{0pt}
$\frn_{10}$ & 3 & 20 & $1357N(1)$ & 7 & 8 \\
\specialrule{1.25pt}{0pt}{0pt}
$147E1(2)$ & 6 & 10 & $1357S(-3)$ & 7 & 8  \\
\specialrule{1pt}{0pt}{0pt}
$247A$ & 4 & 13 & $12457H$ & 7 & 6\\
\specialrule{1pt}{0pt}{0pt}
$247L$ & 6 & 13 & $12457I$ & 8 & 6\\
\bottomrule[1pt]
\end{tabu}}
\end{table} 

Unfortunately, we have not been able to compute $\dim\calM^{\mathrm{c}}(\frn_4)$. By looking at Table~\ref{Table:2} we deduce the following theorem:
\begin{theorem}
On 7-dimensional non-abelian nilpotent Lie algebras with a closed $\G_2$-structure there is no relation between the dimension of $\calM^{\mathrm{c}}(\frg)$ and $b_3(\frg)$. More precisely:
\begin{itemize}
    \item $\dim\calM^{\mathrm{c}}(\frg)<b_3(\frg)$ if $\frg$ is decomposable;
    \item both $\dim\calM^{\mathrm{c}}(\frg)<b_3(\frg)$ and $\dim\calM^{\mathrm{c}}(\frg)>b_3(\frg)$ are possible on indecomposable Lie algebras.
\end{itemize}
\end{theorem}

\begin{corollary}
On a 7-dimensional nilmanifold $M=\Gamma\backslash G$ endowed with left-invariant closed $\G_2$-structures there is no relation between the dimension of $\calM^{\mathrm{c}}(\frg)$ and $b_3(M)$.    
\end{corollary}

\begin{remark}
    For nilpotent Lie algebras equipped with closed $\G_2$-structures, there is a relation between the dimension of the moduli space $\calM^{\mathrm{c}}(\frg)$ and the nilpotency step of $\frg$. Indeed, $\dim\calM^{\mathrm{c}}(\frg)<b_3(\frg)$ if $\frg$ has nilpotency step $\leq4$ and $\dim\calM^{\mathrm{c}}(\frg)>b_3(\frg)$ for $\frg$ 5-step nilpotent.
\end{remark}

\begin{remark}
    Although there is no classification of closed $\G_2$-structures on 7-dimensional solvable Lie algebras, the same kind of computations can be carried out in the solvable setting. As an example, consider the pair $(\frg,\varphi)$ from \cite{Fer87}, where 
    \begin{itemize}
        \item $\frg=(0,0,13-24,14+23,-15+26,-16-25,0)$ is a 2-step solvable Lie algebra, and
        \item $\varphi=e^{134}+e^{467}+e^{156}-e^{127}-e^{357}+e^{245}+e^{236}$ is a closed $\G_2$-structure on $\frg$.
    \end{itemize} In this case we have $\dim V^{\mathrm{c}}(\frg)=19$, $\dim\Aut(\frg)=11$ and $\dim\Aut(\frg)_\varphi=1$. If we consider the perturbation $\varphi_\varepsilon=\varphi+\varepsilon e^{123}$ we have $\dim\Aut(\frg)_{\varphi_\varepsilon}=0$. Hence $\dim\calM^{\mathrm{c}}(\frg)=8>7=b_3(\frg)$. 
\end{remark}

\section{Moduli spaces of coclosed $\G_2$-structures on NLAs}\label{Sec:4}

In this section we focus on moduli spaces of coclosed $\G_2$-structures, using the same approach as in Section~\ref{Sec:3}. We consider three moduli spaces of coclosed $\G_2$-structures:
\begin{itemize}
    \item coclosed $\G_2$-structures on (nilpotent) Lie algebras,
    \item left-invariant coclosed $\G_2$-structures on nilmanifolds,
    \item coclosed $\G_2$-structures on 7-manifolds.
\end{itemize}

Given a 7-dimensional NLA $\frg$ with a coclosed $\G_2$-structure, we set
\[
V^{\tr{cc}}(\frg)\coloneqq \left\{\psi\in\Lambda^4\frg^*\mid \text{$\psi$ determines a $\G_2$-structure and } \d\psi=0\right\}\,.
\]
Observe that starting with a 3-form $\varphi$ defining a $\G_2$-structure, one has $\psi=\ast_\varphi\varphi$. The stabilizer of such 4-form is $\pm\G_2$; therefore, the 4-form $\psi$ alone does determine the $\G_2$-metric, but not the orientation. An explicit formula, which gives the metric directly from $\psi$, can be found in \cite[Page~5]{FMMR23}. Notice that $V^{\tr{cc}}(\frg)$ is an open subset of $Z^4(\frg^*)$, hence its dimension coincides with that of $Z^4(\frg^*)$. $\Aut(\frg)$ acts naturally on $V^{\tr{cc}}(\frg)$. We denote by $\calO(\psi)\cong \Aut(\frg)/\Aut(\frg)_\psi$ the orbit of $\psi\in V^{\tr{cc}}(\frg)$ under the $\Aut(\frg)$-action; here $\Aut(\frg)_\psi$ is the stabilizer of $\psi$; notice that $\Aut(\frg)_\psi$ is compact.

\begin{definition}
The {\em moduli space of coclosed $\G_2$-structures on $\frg$} is
\[
 \calM^{\tr{cc}}(\frg)\coloneqq V^{\tr{cc}}(\frg)/\Aut(\frg)=\{\calO(\psi)\mid\psi\in V^{\tr{cc}}(\frg)\}\,.
\]
\end{definition}
As in the closed case, it is a stratified space, since the orbits may have different dimensions. 

\begin{definition}
The {\em dimension} of $\calM^{\tr{cc}}(\frg)$ is 
\begin{equation}\label{eq:dimension:coc}
\dim\calM^{\tr{cc}}(\frg)\coloneqq\dim V^{\tr{cc}}(\frg)-\dim\calO(\psi)=\dim V^{\tr{cc}}(\frg)-(\dim\Aut(\frg)-\dim\Aut(\frg)_\psi)\,,    
\end{equation}
where $\calO(\psi)$ is a principal orbit; equivalently, $\psi$ has minimal stabilizer.    
\end{definition}

If $\Gamma\backslash G$ is a nilmanifold, the {\em moduli space of left-invariant coclosed $\G_2$-structures} on $\Gamma\backslash G$ is 
\[
\calM^{\tr{cc}}_{\mathrm{inv}}(\Gamma\backslash G)\coloneqq V^{\tr{cc}}_{\mathrm{inv}}(\Gamma\backslash G)/\Aut(G,\Gamma)=\{\calO(\psi_{\mathrm{inv}})\mid\psi_{\mathrm{inv}}\in V^{\tr{cc}}_{\mathrm{inv}}(\Gamma\backslash G)\}\,,
\]
where $V^{\tr{cc}}_{\mathrm{inv}}(\Gamma\backslash G)$ is the space of left-invariant coclosed $\G_2$-structures on $\Gamma\backslash G$, thus $V^{\tr{cc}}_{\mathrm{inv}}(\Gamma\backslash G)\cong V^{\tr{cc}}(\frg)$. Arguing exactly as in Section~\ref{Sec:3}, we see that on a nilmanifold $\Gamma\backslash G$ equipped with a left-invariant coclosed $\G_2$-structure we have 
\[
\dim\calM^{\tr{cc}}_{\mathrm{inv}}(\Gamma\backslash G)=\dim V^{\tr{cc}}(\frg)=\dim\left(\ker\left(\d\colon\Lambda^4\frg^*\to\Lambda^5\frg^*\right)\right)\geq b_4(\frg)=b_3(\frg)=b_3(\Gamma\backslash G)\,.
\]

\begin{definition}
Let $M$ be a manifold endowed with coclosed $\G_2$-structures. The {\em moduli space of coclosed $\G_2$-structures on $M$} is
\[
\calM^{\tr{cc}}(M)\coloneqq\{\psi\in\Omega^4_+(M), \d\psi=0\}/\Diff_0(M)=\{\calO(\psi) \mid \psi\in\Omega^4_+(M), \d\psi=0\}\,,
\]
where $\calO(\psi)$ is the orbit of $\psi$ under the action of $\Diff_0(M)$, so that $\calO(\psi)\cong\Diff_0(M)/\Diff_0(M)_{\psi}$.
\end{definition}

We also have:

\begin{proposition}
Let $M$ be a compact manifold endowed with coclosed $\G_2$-structures. Then $\calM^{\tr{cc}}(M)$ is infinite-dimensional.
\end{proposition}

\begin{proof}
 $\{\psi\in\Omega^4_+(M), \d\psi=0\}$ is an open set in the infinite-dimensional vector space $Z^4(M)$ of closed 4-forms on $M$, acted on by $\Diff_0(M)$. Pick $[\psi]\in\calM^{\tr{cc}}(M)$. Then
 \[
 T_{[\psi]}\calM^{\tr{cc}}(M)=Z^4(M)/T_\psi(\Diff_0(M)\cdot\psi)\,,
 \]
 where $T_\psi(\Diff_0(M)\cdot\psi)=\{\mathcal L_X\psi \mid X\in\frX(M)\}$. Since $\d\psi=0$,
 \[
 T_\psi(\Diff_0(M)\cdot\psi)=\{\d(\iota_X\psi) \mid X\in\frX(M)\}=\d\left(\Omega^3_7(M)\right)\,,
 \]
 where $\Omega^3(M)=\langle\varphi\rangle\oplus\Omega^3_7(M)\oplus\Omega^3_{27}(M)$ under the natural action of $\G_2$. By harmonic theory, $Z^4(M)=\mathcal{H}^4(M)\oplus \d\left(\Omega^3(M)\right)$, hence $T_{[\psi]}\calM^{\tr{cc}}(M)\cong \mathcal{H}^4(M)\oplus \d\left(\langle \varphi\rangle\oplus\Omega^3_{27}(M)\right)$.
\end{proof}

\begin{remark}
Not much seems to be known about the group $\Diff_0(M)_\psi$, except that it is non-trivial, in general. Below, we will show that it needs not be abelian.
\end{remark}


If $M=\Gamma\backslash G$ is a nilmanifold admitting left-invariant coclosed $\G_2$-structures, we have a natural inclusion $\imath\colon V^{\tr{cc}}(\frg)\hookrightarrow\{\psi\in\Omega^4_+(M) \mid \d\psi=0\}$. Arguing as in Theorem~\ref{thm:closed}, we obtain:

\begin{theorem}\label{thm:coclosed}
    $\imath$ induces an inclusion $\bar{\imath}\colon \calM^{\tr{cc}}(\frg)\hookrightarrow\calM^{\tr{cc}}(M)$.
\end{theorem}

As we did in the previous section, we can ask whether, on a nilmanifold $\Gamma\backslash G$ admitting coclosed $\G_2$-structures, there is any relationship between $\dim \calM^{\tr{cc}}(\frg)$ and $b_4(\Gamma\backslash G)=b_3(\Gamma\backslash G)=b_3(\frg)$. We address this question next.


By \eqref{eq:dimension:coc} to compute $\dim \calM^{\tr{cc}}(\frg)$ we need the dimensions of $ V^{\tr{cc}}(\frg)$, $\Aut(\frg)$ and $\Aut(\frg)_\psi$.

\begin{itemize}
    \item We noticed above that $\dim V^{\tr{cc}}(\frg)=\dim Z^4(\frg^*)$;
    \item $\dim\Aut(\frg)$ is computed by the dimension of it Lie algebra $\Der(\frg)=\{D\in\mathfrak{gl}(\frg) \mid D[X,Y]=[DX,Y]+[X,DY] \ \forall~X,Y\in\frg\}$;
    \item $\dim\Aut(\frg)_\psi$ is computed by the dimension of its Lie algebra $\Der(\frg)_\psi=\{D\in\Der(\frg) \mid D\cdot\psi=0\}$, where, for $X,Y,Z,W\in\frg$,
\begin{align*}
(D\cdot \psi)(X,Y,Z,W)=&-\psi(DX,Y,Z,W)-\psi(X,DY,Z,W)\\
&-\psi(X,Y,DZ,W)-\psi(X,Y,Z,DW)\,.
\end{align*}
\end{itemize}

To compute the dimension of $\Aut(\frg)_\psi$ we need a background coclosed $\G_2$-structure on each of these Lie algebras. 

In \cite{BFF18}, Bagaglini, Fino and Fernández classified the 7-dimensional decomposable and 2-step indecomposable nilpotent Lie algebras admitting a coclosed $\G_2$-structure. According to \cite{Gong}, there are 35 decomposable and 9 indecomposable 2-step 7-dimensional NLAs and, as shown in \cite{BFF18}, only 24 of the decomposable and 7 of the 2-step admit a coclosed $\G_2$-structures.

For an explicit coclosed $\G_2$-structure we look at \cite{BGM21}, where the authors construct a (purely) coclosed $\G_2$-structure on a 7-dimensional Lie algebra $\frg$ starting from a particular kind of $\mathrm{SU}(3)$-structure $(\omega,\psi_-)$ on a codimension 1 subspace and a covector $\eta\in\frg^*$ which pairs non-trivially with the center of $\frg$, see \cite[Theorem~4.1]{BGM21}. With these ingredients, the $\G_2$ 4-form on $\frg$ is
\[
\psi=\frac{1}{2}\omega^2+\psi_-\wedge\eta\,.
\]

\begin{remark}
For the Lie algebra $\frn_2$ we used the coclosed $\G_2$-form $\psi$ constructed in \cite{BFF18}. The reason for this is that the ones constructed in \cite{BGM21} are {\em purely} coclosed $\G_2$-structures, namely coclosed $\G_2$-structures for which $\varphi\wedge \d\varphi=0$ (here $\psi=*_\varphi\varphi$). However, combining \cite[Theorem~4.3, Theorem~5.1]{BFF18} and \cite[Theorem~1.4]{dBMR22} one sees that $\frn_2$ is the only 7-dimensional 2-step nilpotent Lie algebra which admits coclosed $\G_2$-structures but no purely coclosed $\G_2$-structures. 
\end{remark}

With this background coclosed $\G_2$-structure $\psi$ we can compute the dimension of the stabilizer $\Aut(\frg)_\psi$, for each of the pairs $(\frg,\psi)$. The results are contained in the following tables (see Appendix~\ref{appendix:2} for the explicit coclosed $\G_2$-form we use and for an overview of the computations).

\begin{table}[H]
\centering
\caption{Decomposable nilpotent Lie algebras with coclosed $\G_2$-structures - I}\label{Table:3_1}
\vspace{0.25 cm}
{\tabulinesep=1.2mm
\begin{tabu}{c|c|c|c|c|c|}
\toprule[1.5pt]
\multirow{2}{*}{$\frg$} & nilpotency & \multirow{2}{*}{structure equations} & $\dim $ & $\dim$ & $\dim$\\
 & step & & $V^{\tr{cc}}(\frg)$ & $\Aut(\frg)$ & $\Aut(\frg)_\psi$\\
\specialrule{1pt}{0pt}{0pt}
$\frn_1 $ & 1 & $(0,0,0,0,0,0,0)$ & 35 & 49 & 14\\
\specialrule{1pt}{0pt}{0pt}
$\frn_2$ & 2 & $(0,0,0,0,0,12,0)$ & 31 & 34 & 4\\
\specialrule{1pt}{0pt}{0pt}
$\frn_3$ & 2 & $(0,0,0,0,0,12+34,0)$ & 29 & 29 & 2\\
\specialrule{1pt}{0pt}{0pt}
$\frn_4$ & 2 & $(0,0,0,0,12,13,0)$ & 29 & 27 & 0\\
\specialrule{1pt}{0pt}{0pt}
$\frn_5$ & 2 & $(0,0,0,0,12,34,0)$ & 28 & 23 & 1\\
\specialrule{1pt}{0pt}{0pt}
$\frn_6$ & 2 & $(0,0,0,0,13-24,14+23,0)$ & 28 & 23 & 0\\
\specialrule{1pt}{0pt}{0pt}
$\frn_7$ & 2 & $(0,0,0,0,12,14+23,0)$ & 28 & 24 & 0\\
\specialrule{1pt}{0pt}{0pt}
$\frn_8$ & 3 &$(0,0,0,0,12,14+25,0)$ & 27 & 22 & 0\\
\specialrule{1pt}{0pt}{0pt}
$\frn_9$ & 3 & $(0,0,0,0,12,15+34,0)$ & 26 & 20 & 0\\
\specialrule{1pt}{0pt}{0pt}
$\frn_{10}$ & 2 & $(0,0,0,12,13,23,0)$ & 28 & 25 & 0\\
\specialrule{1pt}{0pt}{0pt}
$\frn_{11}$ & 3 & $(0,0,0,12,13,24,0)$ & 26 & 19 & 0\\


\bottomrule[1pt]
\end{tabu}}
\end{table}

\begin{table}[H]
\centering
\caption{Decomposable nilpotent Lie algebras with coclosed $\G_2$-structures - II}\label{Table:3_2}
\vspace{0.25 cm}
{\tabulinesep=1.2mm
\begin{tabu}{c|c|c|c|c|c|}
\toprule[1.5pt]
\multirow{2}{*}{$\frg$} & nilpotency & \multirow{2}{*}{structure equations} & $\dim $ & $\dim$ & $\dim$\\
 & step & & $V^{\tr{cc}}(\frg)$ & $\Aut(\frg)$ & $\Aut(\frg)_\psi$\\
\specialrule{1pt}{0pt}{0pt}
$\frn_{12}$ & 3 & $(0,0,0,12,13,14+23,0)$ & 26 & 20 & 0\\
\specialrule{1pt}{0pt}{0pt}
$\frn_{13}$ & 3 & $(0,0,0,12,23,14+35,0)$ & 25 & 16 & 0\\
\specialrule{1pt}{0pt}{0pt}
$\frn_{14}$ & 3 & $(0,0,0,12,23,14-35,0)$ & 25 & 16 & 0\\
\specialrule{1pt}{0pt}{0pt}
$\frn_{15}$ & 3 & $(0,0,0,12,13,14+35,0)$ & 25 & 17 & 0\\
\specialrule{1pt}{0pt}{0pt}
$\frn_{16}$ & 4 & $(0,0,0,12,14,15,0)$ & 25 & 19 & 0\\
\specialrule{1pt}{0pt}{0pt}
$\frn_{17}$ & 4 & $(0,0,0,12,14,15+24,0)$ & 25 & 18 & 0\\
\specialrule{1pt}{0pt}{0pt}
$\frn_{18}$ & 4 & $(0,0,0,12,14,15+24+23,0)$ & 25 & 16 & 0\\
\specialrule{1pt}{0pt}{0pt}
$\frn_{19}$ & 4 & $(0,0,0,12,14,15+23,0)$ & 25 & 17 & 0\\
\specialrule{1pt}{0pt}{0pt}
$\frn_{20}$ & 4 & $(0,0,0,12,14-23,15+34,0)$ & 24 & 15 & 0\\
\specialrule{1pt}{0pt}{0pt}
$\frn_{21}$ & 4 & $(0,0,12,13,23,14+25,0)$ & 24 & 14 & 0\\
\specialrule{1pt}{0pt}{0pt}
$\frn_{22}$ & 4 & $(0,0,12,13,23,14-25,0)$ & 24 & 14 & 0\\
\specialrule{1pt}{0pt}{0pt}
$\frn_{23}$ & 4 & $(0,0,12,13,23,14,0)$ & 24 & 16 & 0\\
\specialrule{1pt}{0pt}{0pt}
$\frn_{24}$ & 5 & $(0,0,12,13,14+23,15+24,0)$ & 23 & 13 & 0\\
\bottomrule[1pt]
\end{tabu}}
\end{table}


\begin{table}[H]
\centering
\caption{Indecomposable 2-step nilpotent Lie algebras with coclosed $\G_2$-structures}\label{Table:4}
\vspace{0.25 cm}
{\tabulinesep=1.2mm
\begin{tabu}{c|c|c|c|c|}
\toprule[1.5pt]
\multirow{2}{*}{$\frg$} & \multirow{2}{*}{structure equations} & $\dim $ & $\dim$ & $\dim$\\
 &  & $V^{\tr{cc}}(\frg)$ & $\Aut(\frg)$ & $\Aut(\frg)_\psi$\\
\specialrule{1.25pt}{0pt}{0pt}
$17$ & $(0,0,0,0,0,0,12+34+56)$ & 29 & 28 & 4\\
\specialrule{1pt}{0pt}{0pt}
$37A$ & $(0,0,0,0,12,23,24)$ & 29 & 25 & 0\\
\specialrule{1pt}{0pt}{0pt}
$37B$ & $(0,0,0,0,12,23,34)$ & 28 & 20 & 0\\
\specialrule{1pt}{0pt}{0pt}
$37B_1$ & $(0,0,0,0,12-34,13+24,14)$ & 28 & 20 & 0\\
\specialrule{1pt}{0pt}{0pt}
$37C$ & $(0,0,0,0,12+34,23,24)$ & 28 & 23 & 0\\
\specialrule{1pt}{0pt}{0pt}
$37D$ & $(0,0,0,0,12+34,13,24)$ & 28 & 19 & 0\\
\specialrule{1pt}{0pt}{0pt}
$37D_1$ & $(0,0,0,0,12-34,13+24,14-23)$ & 28 & 19 & 3\\
\bottomrule[1pt]
\end{tabu}}
\end{table}

\begin{remark}
As it happened in the closed case, if $\frg=\frn_1$, the abelian Lie algebra, the automorphism group acts transitively on the space of (closed) 4-forms; the stabilizer of a point is isomorphic to $\pm\G_2$, and the moduli space is just one point.    
\end{remark}

\begin{remark}
The dimension of $V^{\textrm{cc}}(\frg)$ in Tables \ref{Table:3_1}, \ref{Table:3_2} and \ref{Table:4} coincides with the dimension of $\mathcal{M}^{\textrm{cc}}_{\textrm{inv}}(M)$, where $M$ is any nilmanifold quotient of $G$. 
\end{remark}

As we see from Tables \ref{Table:3_1}, \ref{Table:3_2} and \ref{Table:4}, most of the pairs $(\frg,\psi)$ have $\dim\Aut(\frg)_\psi=0$; therefore, the stabilizer is minimal, the orbit $\calO(\psi)$ has maximal dimension and we can directly compute the dimension of the corresponding moduli space. This is not the case, however, for $\frg\in\{\frn_2,\frn_3,\frn_5,17,37D_1\}$. As in the closed setting, we try to slightly perturb the given $\psi$ to a coclosed $\G_2$ form $\psi_\varepsilon$, for $\varepsilon$ sufficiently small, whose stabilizer has dimension 0.

\begin{itemize}
	\item In the Lie algebras $\frn_2$ and $\frn_3$, the given $\psi$ seems to have minimal stabilizer.
    \item In the Lie algebra $\frn_{5}$ the orbit of the given $\psi$ has 1-dimensional stabilizer. Considering the coclosed $\G_2$ 4-form $\psi_\varepsilon=\psi+\varepsilon(e^{2356}+e^{2457})$, one computes that $\dim\Aut(\frn_5)_{\psi_\varepsilon}=0$.
    \item In the Lie algebra $\frg=17$ the orbit of the given $\psi$ has 4-dimensional stabilizer. Considering the coclosed $\G_2$ 4-form 
    $\psi_\varepsilon=\psi+\varepsilon e^{1234}$, one computes that $\dim\Aut(\frg)_{\psi_\varepsilon}=2$. This dimension seems to be minimal.
	\item In the Lie algebra $\frg=37D_1$ the orbit of the given $\psi$ has 3-dimensional stabilizer. One checks that the coclosed $\G_2$ 4-form $\psi_\varepsilon=\psi+\varepsilon(e^{2367}+e^{3456})$ has $\dim\Aut(\frg)_{\psi_\varepsilon}=0$.
\end{itemize}

All in all, we can compute the dimension of $\calM^{\mathrm{cc}}(\frg)$ for most of the 7-dimensional decomposable and 2-step indecomposable nilpotent Lie algebras endowed with a coclosed $\G_2$ structure. We collect the results in the following tables, together with the third Betti number of each Lie algebra; we remark that since a nilpotent Lie algebra is unimodular, its Lie algebra cohomology satisfies Poincaré duality.

\begin{table}[H]
\centering
\caption{Moduli spaces of coclosed $\G_2$-structures on decomposable Lie algebras}\label{Table:5}
\vspace{0.25 cm}
{\tabulinesep=1.2mm
\begin{tabu}{c|c|c||c|c|c||c|c|c|}
\toprule[1.5pt]
$\frg$ & $\dim\calM^{\mathrm{cc}}(\frg)$ & $b_3(\frg)$ & $\frg$ & $\dim\calM^{\mathrm{cc}}(\frg)$ & $b_3(\frg)$ & $\frg$ & $\dim\calM^{\mathrm{cc}}(\frg)$ & $b_3(\frg)$\\
\specialrule{1pt}{0pt}{0pt}
$\frn_1$ & 0 & 35 & $\frn_9$ & 6 & 12 & $\frn_{17}$ & 7 & 11\\
\specialrule{1pt}{0pt}{0pt}
$\frn_2$ & $\leq 1$ & 25 & $\frn_{10}$ & 3 & 20 & $\frn_{18}$ & 9 & 11\\
\specialrule{1pt}{0pt}{0pt}
$\frn_3$ & $\leq 2$ & 19 & $\frn_{11}$ & 7 & 14 & $\frn_{19}$ & 8 & 11\\
\specialrule{1pt}{0pt}{0pt}
$\frn_4$ & 2 & 21 & $\frn_{12}$ & 6 & 14 & $\frn_{20}$ & 9 & 8\\
\specialrule{1pt}{0pt}{0pt}
$\frn_5$ & 5 & 18 & $\frn_{13}$ & 9 & 11 & $\frn_{21}$ & 10 & 10\\
\specialrule{1pt}{0pt}{0pt}
$\frn_6$ & 4 & 18 & $\frn_{14}$ & 9 & 11 & $\frn_{22}$ & 10 & 10\\
\specialrule{1pt}{0pt}{0pt}
$\frn_7$ & 5 & 18 & $\frn_{15}$ & 8 & 11 & $\frn_{23}$ & 8 & 10\\
\specialrule{1pt}{0pt}{0pt}
$\frn_8$ & 6 & 15 & $\frn_{16}$ & 6 & 11 & $\frn_{24}$ & 10 & 7\\
\bottomrule[1pt]
\end{tabu}}
\end{table}

\begin{table}[H]
\centering
\caption{Moduli spaces of coclosed $\G_2$-structures on indecomposable Lie algebras}\label{Table:6}
\vspace{0.25 cm}
{\tabulinesep=1.2mm
\begin{tabu}{c|c|c||c|c|c|}
\toprule[1.5pt]
$\frg$ & $\dim\calM^{\mathrm{cc}}(\frg)$ & $b_3(\frg)$ & $\frg$ & $\dim\calM^{\mathrm{cc}}(\frg)$ & $b_3(\frg)$\\
\specialrule{1.25pt}{0pt}{0pt}
$17$ & $\leq3$ & 14 & $37C$ & 5 & 17\\
\specialrule{1pt}{0pt}{0pt}
$37A$ & 4 & 18 & $37D$ & 9 & 14\\
\specialrule{1pt}{0pt}{0pt}
$37B$ & 8 & 16 & $37D_1$ & 9 & 14\\
\specialrule{1pt}{0pt}{0pt}
$37B_1$ & 8 & 16\\
\cmidrule[1.25pt]{1-3}
\end{tabu}}
\end{table}

Unfortunately, we have not been able to compute the dimensions of the moduli spaces for the Lie algebras $\frn_2$, $\frn_3$ and $17$. We can isolate the following consequences of Tables \ref{Table:5} and \ref{Table:6}:

\begin{theorem}
On a 7-dimensional non-abelian nilpotent Lie algebra $\frg$ with a coclosed $\G_2$-structure there is no relation between the dimensions of $\calM^{\mathrm{cc}}(\frg)$ and $b_3(\frg)=b_4(\frg)$. More precisely:
\begin{itemize}
    \item $\dim\calM^{\mathrm{cc}}(\frg)<b_3(\frg)$, $\dim\calM^{\mathrm{cc}}(\frg)=b_3(\frg)$ and $\dim\calM^{\mathrm{cc}}(\frg)>b_3(\frg)$ are possible on decomposable Lie algebras;
    \item $\dim\calM^{\mathrm{cc}}(\frg)<b_3(\frg)$ if $\frg$ is 2-step indecomposable.
\end{itemize}
\end{theorem}

\begin{remark}
    In contrast with the closed case, for Lie algebras with coclosed $\G_2$-structure we do not have a relation between the dimension of the moduli space and the nilpotency step of $\frg$. For example, the Lie algebras $\frn_{20}$, $\frn_{21}$ and $\frn_{23}$ are 4-step nilpotent, but $$\dim\calM^{\tr{cc}}(\frn_{20})>b_3(\frn_{20}),\quad\dim\calM^{\tr{cc}}(\frn_{21})=b_3(\frn_{21}),\quad\dim\calM^{\tr{cc}}(\frn_{23})<b_3(\frn_{23}).$$
\end{remark}

One can wonder whether the uniform bound $\dim \calM^{\mathrm{cc}}(\frg) < b_3(\frg)$ on indecomposable Lie algebras holds because we restrict to nilpotency step 2. This is indeed the case, as the following example shows:

\begin{example}\label{ex:coclosed}
The Lie algebra $\frg=137A=(0,0,0,0,12,34,15+36)$ is 7-dimensional, 3-step nilpotent, has $\dim V^{\tr{cc}}(\frg)=24$, $b_3(\frg)=8$ and $\dim\Aut(\frg)=14$. It carries the (purely) coclosed $\G_2$-structure $\psi=\frac{1}{2}\omega^2+\psi_-\wedge\eta$ with
\begin{itemize}
    \item $\omega=e^{13}-e^{24}+e^{56}$;
    \item $\psi_-=e^{126}-e^{145}-e^{235}-e^{345}+e^{346}$;
    \item $\eta=e^7$.
\end{itemize}
One checks that $\dim\Aut(\frg)_\psi=0$, hence $\dim\calM^{\mathrm{cc}}(\frg)=24-14=10>8=b_3(\frg)$. Similarly, one can check that $\dim\calM^{\tr{cc}}(247Q)=10=b_3(247Q)$ and $\dim\calM^{\tr{cc}}(357A)=8<14=b_3(357A)$, where $247Q$ and $357A$ are also 3-step nilpotent indecomposable Lie algebras carrying the (purely) coclosed $\G_2$-structure described in \cite[Table~4]{BGM21}.
\end{example}

\begin{corollary}
On a 7-dimensional nilmanifold $M=\Gamma\backslash G$ endowed with left-invariant coclosed $\G_2$-structures there is no relation between the dimension of $\calM^{\mathrm{cc}}(\frg)$ and $b_3(M)$.    
\end{corollary}

As a by-product of our study of the automorphism group of coclosed $\G_2$-structures on nilpotent Lie algebras we obtain the following result:

\begin{theorem}
The automorphism group of a coclosed $\G_2$-structure need not be abelian.    
\end{theorem}

\begin{proof}
Consider the decomposable Lie algebra $\frn_2$ with the coclosed $\G_2$-structure $\psi$ given in Appendix \ref{appendix:2}. As we saw in Table~\ref{Table:3_1}, the Lie algebra $\Der(\frn_2)_\psi$ has dimension 4; moreover, it contains $\mathfrak{sl}_2(\R)$ as a subalgebra, hence it is not abelian. Then $\Aut(\frn_2)_\psi$ is also not abelian. Notice that $\frn_2\cong\R^4\oplus\frh_3$, where $\frh_3$ denotes the Heisenberg Lie algebra. One can pick two non-commuting derivations in $\Der(\frn_2)_\psi$ to give the following automorphisms of the abelian factor $\R^4$ in $\R^4\times H_3$:
\[
A_t=\begin{pmatrix}
    \cos t & 0 & \sin t & 0\\
    0 & \cos t & 0 & -\sin t\\
    -\sin t & 0 & \cos t & 0\\
    0 & \sin t & 0 & \cos t
\end{pmatrix} \quad \tr{and} \quad
B_s=\begin{pmatrix}
    \cos s & 0 & 0 & -\sin s\\
    0 & \cos s & -\sin s & 0\\
    0 & \sin s & \cos s & 0\\
    \sin s & 0 & 0 & \cos s
\end{pmatrix}.
\]
Choosing $t=s=\frac{\pi}{2}$, we obtain two non-commuting automorphisms of the Lie group $\R^4$ which preserve the standard lattice $\Z^4\subset\R^4$. Together with the identity on the factor $H_3$, which preserves any given lattice $\Gamma\subset H_3$, they define non-commuting elements in $\Diff(M)_\psi$, where $M$ is the nilmanifold $(\Z^4\times\Gamma)\backslash(\R^4\times H_3)$. 
\end{proof}

\newpage

\appendix

\section{The closed package}\label{appendix:1}

\begin{table}[H]
\centering
\caption{Closed $\G_2$-structures on 7-dimensional nilpotent Lie algebras}\label{Table:7}
\vspace{0.25 cm}
{\tabulinesep=1.2mm
\begin{tabu}{c|c|}
\toprule[1.5pt]
$\frg$ & $\varphi$\\
\specialrule{1pt}{0pt}{0pt}
$\frn_1$ & $e^{127}+e^{347}+e^{567}+e^{135}-e^{146}-e^{236}-e^{245}$\\
\specialrule{1pt}{0pt}{0pt}
$\frn_4$ & $e^{123}+e^{147}+e^{156}+e^{245}+e^{267}-e^{346}+e^{357}$\\
\specialrule{1pt}{0pt}{0pt}
$\frn_{10}$ & $e^{123}+e^{145}+e^{167}-e^{246}+2e^{257}+e^{347}+e^{356}$\\
\specialrule{1.25pt}{0pt}{0pt}
$1357N(1)$ & $\begin{array}{c}
           12\sqrt{3}\left(e^{126}+e^{134}+e^{135}-e^{136}-e^{157}+e^{237}\right.\\
           \left.-e^{245}-e^{267}-e^{356}+e^{467}\right)\end{array}$\\
\specialrule{1pt}{0pt}{0pt}
$1357S(-3)$ & $\begin{array}{c}
		     -14\sqrt{70}\left(2e^{123}+2e^{146}-e^{157}+e^{167}+2e^{245}+2e^{246}\right.\\
		      \left.+e^{257}+3e^{267}+e^{347}+4e^{356}+2e^{456}\right)
		\end{array}$\\
\specialrule{1pt}{0pt}{0pt}
$147E1(2)$ & $\begin{array}{c}
            -6\sqrt{6}\left(2e^{123}-2e^{124}+2e^{135}-e^{157}-e^{167}\right.\\
            \left.+2e^{236}+e^{257}+e^{347}+2e^{456}\right)\end{array}$\\
\specialrule{1pt}{0pt}{0pt}
$257A$ & $e^{124}-e^{135}-e^{167}+e^{236}-e^{257}-e^{347}+e^{456}$\\
\specialrule{1pt}{0pt}{0pt}
$257B$ & $-e^{125}+e^{134}+e^{167}-e^{237}-e^{246}-e^{356}+e^{457}$\\
\specialrule{1pt}{0pt}{0pt}
$247A$ & $e^{123}+e^{145}+e^{167}-e^{246}+e^{257}+e^{347}+e^{356}$\\
\specialrule{1pt}{0pt}{0pt}
$247L$ & $e^{123}+e^{126}-e^{137}-e^{145}-2e^{167}+e^{246}-e^{257}+e^{347}+e^{356}$\\
\specialrule{1pt}{0pt}{0pt}
$12457H$ & $e^{125}+e^{127}+e^{136}+e^{147}+e^{234}-e^{256}-e^{267}+e^{357}-e^{456}$\\
\specialrule{1pt}{0pt}{0pt}
$12457I$ & $\begin{array}{c}
12\sqrt{6}\left(e^{124}-e^{125}-e^{127}-e^{136}-e^{147}-e^{234}\right.\\
\left.+e^{267}-e^{345}-e^{357}+e^{456}\right)
\end{array}$\\
\bottomrule[1pt]
\end{tabu}}
\end{table} 

\begin{remark}
There seems to be a small typo in the construction of the $\G_2$-structure on $\frn_{10}$: to obtain a positive-definite metric, one should take $\psi_2^+=e^{123}+e^{145}-e^{246}+e^{356}$ in the proof of \cite[Theorem 4]{CF11}.
\end{remark}

Next we add, for a particular case, the SageMath 9.3 code used in the article. To compute the Betti numbers, we set the boundary operator of the Lie algebra and then we compute the homology with the next code.

\begin{lstlisting}[language=Python]
############################## Betti numbers of n4
E = ExteriorAlgebra(SR,'e',7);
E.inject_variables();
alg_nil_g2 = {"\\frn_4":{(0,1):e4,(0,2):e5}};
d = E.coboundary(alg_nil_g2["\\frn_4"]);
cplx_dR = d.chain_complex();
cplx_dR.homology()
\end{lstlisting}

To compute the dimension of $V^{\tr{c}}(\frg)$ we compute the dimension of $Z^3(\frg^*)$.

\begin{lstlisting}[language=Python]
############################## Dimension of V^c(n4)
E = ExteriorAlgebra(SR,'e',7);
E.inject_variables();
alg_nil_g2 = {"\\frn_4":{(0,1):e4,(0,2):e5}};
d = E.coboundary(alg_nil_g2["\\frn_4"]);
varphicoeff = var(["varphi%d%d%d"%(i,j,k)\
for i in range(len(E.basis(1)))\
for j in range(i+1, len(E.basis(1)))\
for k in range(j+1, len(E.basis(1)))]);
varphi = sum(varphicoeff[l]*list(E.basis(3))[l] for l in range(len(E.basis(3))));
eq0 = [d(varphi).interior_product(c).constant_coefficient()==0 for c in E.basis(4)];
m = matrix([[lhs.coefficient(x) for x in varphicoeff] for e in eq0 for lhs in [e.lhs()]]);
show(dim(m.right_kernel()));
\end{lstlisting}

The following code compute the dimension of the automorphism group.

\begin{lstlisting}[language=Python]
############################## Dimension of Aut(n4)
alg_nil_g2 = {"\\frn_4":{('e0','e1'):{'e4':-1},('e0','e2'):{'e5':-1}}};
L.<e0,e1,e2,e3,e4,e5,e6> = LieAlgebra(SR,alg_nil_g2["\\frn_4"]);
show(len(L.derivations_basis()));
\end{lstlisting}

Finally, to compute the dimension of the stabilizer, we use the following code.

\begin{lstlisting}[language=Python]
############################## Derivations of n4
alg_nil_g2 = {"\\frn_4":{('e0','e1'):{'e4':-1},('e0','e2'):{'e5':-1}}};
L.<e0,e1,e2,e3,e4,e5,e6> = LieAlgebra(SR,alg_nil_g2["\\frn_4"]);
d = len(L.derivations_basis());
X = [var("x%d"%i) for i in range(d)];
Dmat = sum(X[i]*L.derivations_basis()[i] for i in range(d));
############################## Closed G2-form of n4
E = ExteriorAlgebra(SR,'e',7);
E.inject_variables();
alg_nil_g2 = {"\\frn_4":{(0,1):e4,(0,2):e5}};
d = E.coboundary(alg_nil_g2["\\frn_4"]);
var(["f%d"%(i) for i in range(len(E.basis(1)))]);
omega = e0*e3+e1*e5+e2*e4;
eta = e6;
psi = e0*e1*e2+e0*e4*e5+e1*e3*e4-e2*e3*e5;
varphi = omega*eta+psi;
############################## Dimension stabilizer
D = E.lift_morphism(Dmat);
eq = [];
for i in range(len(E.gens())):
	for j in range(i+1,len(E.gens())):
		for k in range(j+1,len(E.gens())):
			eq1 = (varphi.interior_product(D(E.gens()[i])).interior_product(E.gens()[j]).interior_product(E.gens()[k])).constant_coefficient();
			eq2 = (varphi.interior_product(E.gens()[i]).interior_product(D(E.gens()[j])).interior_product(E.gens()[k])).constant_coefficient();
			eq3 = (varphi.interior_product(E.gens()[i]).interior_product(E.gens()[j]).interior_product(D(E.gens()[k]))).constant_coefficient();
			eq.append(eq1+eq2+eq3);
eq0 = eq;
eq = [eq0[i]==0 for i in range(len(eq))];
m = matrix([[lhs.coefficient(x) for x in X] for e in eq for lhs in [e.lhs()]]);
show(dim(m.right_kernel()));
\end{lstlisting}

\section{The coclosed package}\label{appendix:2}

\begin{table}[H]
\centering
\caption{Coclosed $\G_2$-structures on decomposable nilpotent Lie algebras}\label{Table:8}
\vspace{0.25 cm}
{\tabulinesep=1.2mm
\begin{tabu}{c|c|}
\toprule[1.5pt]
$\frg$ & $\psi$\\
\specialrule{1pt}{0pt}{0pt}
$\frn_1$ & $e^{1234}+e^{1256}+e^{1367}+e^{1457}+e^{2357}-e^{2467}+e^{3456}$\\
\specialrule{1pt}{0pt}{0pt}
$\frn_2$ & $-e^{1234}+e^{1257}+e^{1356}-e^{1467}+e^{2367}+e^{2456}-e^{3457}$\\
\specialrule{1pt}{0pt}{0pt}
$\frn_3$ & $-e^{1234}+e^{1256}+e^{1357}+e^{1467}-e^{2367}+e^{2457}-e^{3456}$\\
\specialrule{1pt}{0pt}{0pt}
$\frn_4$ & $-2e^{1234}-e^{1256}+e^{1357}+3e^{1456}-e^{1467}-2e^{2356}+e^{2367}+e^{2457}+e^{3467}$\\
\specialrule{1pt}{0pt}{0pt}
$\frn_5$ & $-e^{1234}+e^{1257}+e^{1267}-e^{1356}+e^{1467}-e^{2367}-e^{2456}-e^{3457}$\\
\specialrule{1pt}{0pt}{0pt}
$\frn_6$ & $e^{1234}+e^{1256}+e^{1257}+e^{1267}+e^{1467}-e^{2367}+e^{3456}-e^{3457}$\\
\specialrule{1pt}{0pt}{0pt}
$\frn_7$ & $e^{1234}-e^{1257}-e^{1356}-2e^{1467}+e^{2367}-e^{2456}+e^{3457}$\\
\specialrule{1pt}{0pt}{0pt}
$\frn_8$ & $\begin{array}{c}
    2e^{1234}+e^{1245}-2e^{1345}-3e^{1236}-e^{2346}+e^{1256}-6e^{1356}+e^{2456}\\
    +e^{2347}+e^{1357}-e^{3457}+e^{2367}+e^{1467}-e^{2567}
\end{array}$\\
\specialrule{1pt}{0pt}{0pt}
$\frn_9$ & $e^{1234}-e^{1257}-e^{1356}-2e^{1467}+e^{2367}-e^{2456}+e^{3457}$\\
\specialrule{1pt}{0pt}{0pt}
$\frn_{10}$ & $\begin{array}{c}
\tfrac{6}{5}e^{1235}-\tfrac{34}{25}e^{1237}+e^{1256}+\tfrac{12}{25}e^{1257}+\tfrac{8}{5}e^{1345}+e^{1346}\\
-\tfrac{12}{25}e^{1347}-\tfrac{16}{25}e^{1457}+2e^{2345}-\tfrac{6}{5}e^{2347}+e^{2467}-e^{3567}
\end{array}$\\
\specialrule{1pt}{0pt}{0pt}
$\frn_{11}$ & $e^{2345}-e^{1236}+e^{1246}-2e^{1456}-e^{1237}+e^{1357}-e^{1457}-e^{3467}-e^{2567}$\\
\specialrule{1pt}{0pt}{0pt}
$\frn_{12}$ & $\begin{array}{c}
-35e^{1234}+35e^{1245}+7e^{2345}+14e^{1236}-7e^{1346}+7e^{1256}+7e^{2356}-7e^{1456}\\
+e^{1237}+e^{1347}+e^{3457}+e^{1367}+e^{2467}+e^{1567}
\end{array}$\\
\specialrule{1pt}{0pt}{0pt}
$\frn_{13}$ & $-2e^{1245}-2e^{1345}-2e^{2345}-\frac{1}{2}e^{1236}-2e^{2456}-e^{1247}+e^{2357}-e^{3467}+e^{1567}$\\
\specialrule{1pt}{0pt}{0pt}
$\frn_{14}$ & $-e^{1345}-e^{1236}-e^{2456}-e^{1247}+e^{2357}-e^{3467}+e^{1567}$\\
\specialrule{1pt}{0pt}{0pt}
$\frn_{15}$ & $\begin{array}{c}
e^{1234}+e^{1345}+\frac{1}{2}e^{2345}-\frac{1}{2}e^{1346}-e^{2346}+\frac{3}{4}e^{1256}\\[.25cm]
-e^{1237}+e^{2457}+e^{2367}+e^{1467}-e^{3567}
\end{array}$\\
\specialrule{1pt}{0pt}{0pt}
$\frn_{16}$ & $3e^{1235}+2e^{1245}-e^{2345}-e^{1346}-e^{1256}-e^{2456}-e^{1247}+e^{3457}+e^{2367}+e^{1567}$\\
\specialrule{1pt}{0pt}{0pt}
$\frn_{17}$ & $-2e^{1235}-e^{1245}+e^{1346}-e^{2456}-e^{1237}+e^{3457}+e^{2367}+e^{2467}-e^{1567}$\\
\specialrule{1pt}{0pt}{0pt}
$\frn_{18}$ & $\begin{array}{c}
-\frac{1}{2}e^{1234}-\frac{1}{2}e^{1245}+\frac{1}{2}e^{1345}+\frac{1}{2}e^{2345}-e^{1236}+\frac{1}{2}e^{1246}-\frac{1}{2}e^{1346}+e^{2346}\\[1.5pt]
-e^{1356}-\frac{1}{2}e^{2456}+e^{1237}+e^{1257}-e^{2357}-e^{3457}-e^{1267}-e^{2367}+e^{1467}
\end{array}$\\
\specialrule{1pt}{0pt}{0pt}
$\frn_{19}$ & $\begin{array}{c}
-\frac{5}{12}e^{1234}-\frac{1}{3}e^{1235}+\frac{1}{2}e^{1245}+\frac{5}{6}e^{1346}-2e^{2346}+\frac{5}{12}e^{1356}-2e^{1456}\\[1.5pt]
-\frac{3}{5}e^{2456}+e^{1247}+e^{3457}+2e^{2367}-e^{1567}
\end{array}$\\
\specialrule{1pt}{0pt}{0pt}
$\frn_{20}$ & $-3e^{1234}+\frac{3}{2}e^{1345}-\frac{3}{2}e^{1236}-e^{1356}-\frac{3}{4}e^{2456}-e^{1257}+e^{3457}+e^{2367}-e^{1467}$\\
\specialrule{1pt}{0pt}{0pt}
$\frn_{21}$ & $\frac{5}{4}e^{1245}+e^{1346}-\frac{1}{2}e^{2346}-\frac{1}{2}e^{1356}-e^{2356}-e^{1237}-e^{3457}+e^{2467}+e^{1567}-e^{2567}$\\
\specialrule{1pt}{0pt}{0pt}
$\frn_{22}$ & $e^{1245}-e^{2346}-e^{1356}-e^{1237}-e^{1347}-e^{2357}-2e^{3457}+e^{1467}-e^{2567}$\\
\specialrule{1pt}{0pt}{0pt}
$\frn_{23}$ & $e^{1245}-e^{2346}-e^{1356}+e^{1237}+e^{3457}-e^{1467}+e^{2567}$\\
\specialrule{1pt}{0pt}{0pt}
$\frn_{24}$ & $e^{1245}+e^{1346}-e^{2356}+e^{1456}+2e^{1237}-2e^{1257}+2e^{3457}-e^{2467}-e^{1567}$\\
\bottomrule[1pt]
\end{tabu}}
\end{table}


\begin{table}[H]
\centering
\caption{Coclosed $\G_2$-structures on indecomposable 2-step nilpotent Lie algebras}\label{Table:10}
\vspace{0.25 cm}
{\tabulinesep=1.2mm
\begin{tabu}{c|c|}
\toprule[1.5pt]
$\frg$ & $\psi$\\
\specialrule{1.25pt}{0pt}{0pt}
$17$ & $-\tfrac{1}{2}e^{1234}+e^{1256}+e^{1357}+e^{1467}-e^{2367}+e^{2457}-\tfrac{1}{2}e^{3456}$\\
\specialrule{1pt}{0pt}{0pt}
$37A$ & $-e^{1234}+e^{1256}+e^{1356}+e^{1367}-e^{1457}+e^{2357}+e^{2456}+e^{2467}-e^{3456}$\\
\specialrule{1pt}{0pt}{0pt}
$37B$ & $-e^{1234}-e^{1257}-e^{1367}+e^{1456}-e^{1467}-e^{2356}+e^{2367}-e^{2467}+e^{3457}$\\
\specialrule{1pt}{0pt}{0pt}
$37B_1$ & $2e^{1234}-e^{1267}+2e^{1357}-2e^{1456}-2e^{2356}-2e^{2457}-2e^{3467}$\\
\specialrule{1pt}{0pt}{0pt}
$37C$ & $-e^{1234}-e^{1267}-e^{1356}+e^{1457}-e^{2357}-e^{2456}+e^{3467}$\\
\specialrule{1pt}{0pt}{0pt}
$37D$ & $-e^{1234}+e^{1267}-e^{1356}+e^{1357}+e^{1456}-2e^{1457}+e^{2357}-e^{2456}-e^{3467}$\\
\specialrule{1pt}{0pt}{0pt}
$37D_1$ & $e^{1234}+e^{1267}-e^{1357}-e^{1456}-e^{2356}+e^{2456}+e^{3467}$\\
\bottomrule[1pt]
\end{tabu}}
\end{table}

We include here, again for a particular example, the codes to compute the data for coclosed $\G_2$-structures. The codes to compute the Betti numbers and the dimension of the automorphism group are the same, but changing the Lie algebra. To compute the dimension of $V^{\tr{cc}}(\frg)$ we use the following code that computes the dimension of $Z^4(\frg^*)$.

\begin{lstlisting}[language=Python]
############################## Dimension of V^cc(n3)
E = ExteriorAlgebra(SR,'e',7);
E.inject_variables();
Two_step_7nil = {"\\frn_3":{(0,1):e5,(2,3):e5}};
d = E.coboundary(Two_step_7nil["\\frn_3"]);
psicoeff = var(["psi%d%d%d%d"%(i,j,k,l)\
for i in range(len(E.basis(1)))\
for j in range(i+1,len(E.basis(1)))\
for k in range(j+1,len(E.basis(1)))\
for l in range(k+1,len(E.basis(1)))]);
psi = sum(psicoeff[m]*list(E.basis(4))[m] for m in range(len(E.basis(4))));
eq0 = [d(psi).interior_product(c).constant_coefficient()==0 for c in E.basis(5)];
m = matrix([[lhs.coefficient(x) for x in psicoeff] for e in eq0 for lhs in [e.lhs()]]);
show(dim(m.right_kernel()));
\end{lstlisting}

To conclude, the next code compute the dimension of the stabilizer for some closed 4-form.

\begin{lstlisting}[language=Python]
############################## Derivations of n3
Two_step_7nil = {"\\frn_3":{('e0','e1'):{'e5':-1},('e2','e3'):{'e5':-1}}};
L.<e0,e1,e2,e3,e4,e5,e6> = LieAlgebra(SR,Two_step_7nil["\\frn_3"]);
d = len(L.derivations_basis());
X = [var("x%d"%i) for i in range(d)];
Dmat = sum(X[i]*L.derivations_basis()[i] for i in range(d));
############################## Coclosed G2-form of n3
E = ExteriorAlgebra(SR,'e',7);
E.inject_variables();
Two_step_7nil = {"\\frn_3":{(0,1):e5,(2,3):e5}};
d = E.coboundary(Two_step_7nil["\\frn_3"]);
omega = e0*e2+e1*e3+e4*e5;
psi = e0*e1*e4-e0*e3*e5+e1*e2*e5-e2*e3*e4;
eta = e6;
starphi = (1/2)*(omega*omega)+psi*eta;
############################## Dimension stabilizer
D = E.lift_morphism(Dmat);
eq = [];
for i in range(len(E.gens())):
	for j in range(i+1,len(E.gens())):
		for k in range(j+1,len(E.gens())):
			for l in range(k+1,len(E.gens())):
				eq1 = (starphi.interior_product(D(E.gens()[i])).interior_product(E.gens()[j]).interior_product(E.gens()[k]).interior_product(E.gens()[l])).constant_coefficient();
				eq2 = (starphi.interior_product(E.gens()[i]).interior_product(D(E.gens()[j])).interior_product(E.gens()[k]).interior_product(E.gens()[l])).constant_coefficient();
				eq3 = (starphi.interior_product(E.gens()[i]).interior_product(E.gens()[j]).interior_product(D(E.gens()[k])).interior_product(E.gens()[l])).constant_coefficient();
				eq4 = (starphi.interior_product(E.gens()[i]).interior_product(E.gens()[j]).interior_product(E.gens()[k]).interior_product(D(E.gens()[l]))).constant_coefficient();
				eq.append(eq1+eq2+eq3+eq4);
eq0 = eq;
eq = [eq0[i]==0 for i in range(len(eq))];
m = matrix([[lhs.coefficient(x) for x in X] for e in eq for lhs in [e.lhs()]]);
show(dim(m.right_kernel()));
\end{lstlisting}

\bibliographystyle{plain}
\bibliography{biblio}

\end{document}